\documentclass[a4paper]{article}
\usepackage[english]{babel}
\usepackage[utf8x]{inputenc}
\usepackage[T1]{fontenc}

\usepackage[a4paper,top=3cm,bottom=2cm,left=3cm,right=3cm,marginparwidth=1.75cm]{geometry}

\usepackage{graphicx, amsmath, amssymb, amsthm, amscd}
\usepackage{color}
\usepackage{epsf}
\usepackage{enumerate}

\usepackage{amsmath}
\usepackage{graphicx}
\usepackage[colorinlistoftodos]{todonotes}
\usepackage[colorlinks=true, allcolors=blue]{hyperref}

\newtheorem{theorem}{Theorem}[section]
\newtheorem{lemma}[theorem]{Lemma}
\newtheorem{proposition}[theorem]{Proposition}
\newtheorem{claim}[theorem]{Claim}

\newtheorem{question}{Question}
\newtheorem{example}{Example}
\newtheorem{conjecture}{Conjecture}
\newtheorem{remark}{Remark}
\DeclareMathOperator{\cl}{cl}

\title{Continuously homogeneous hereditarily indecomposable continua are tree-like}
\begin{document}
\maketitle

\begin{center}Jan P. Boro\'nski\footnote{Faculty of Mathematics and Computer Science, Jagiellonian University in Krak\'ow,
ul. Łojasiewicza 6, 30-348 Kraków, Poland, e-mail: jan.boronski@uj.edu.pl}$^,$\footnote{National Supercomputing Centre IT4Innovations, Division of the University of Ostrava, Institute for Research and Applications of Fuzzy Modeling, 30. dubna 22, 70103 Ostrava, Czech Republic, e-mail: jan.boronski@osu.cz}, David R. Prier\footnote{Gannon University, Mathematics Department, 109 University Square, Erie, PA 16541, USA, e-mail: prier001@gannon.edu}, Michel Smith\footnote{Department of Mathematics and Statistics, Auburn University, Auburn, AL 36849, United States, e-mail: smith01@auburn.edu } and Frank Sturm\footnote{Sadly, Frank passed away on Oct. 14, 2014, but the research started in the Spring of 2013, when he and the first author worked at the Department of Mathematics and Statistics at Auburn University; see \cite{Frank}}
\end{center}
\begin{abstract}
A topological space $X$ is continuously homogeneous if for any $x,y\in X$ there exists a continuous surjection $f:X\to X$ with $f(x)=y$. We show that continuously homogeneous hereditarily indecomposable continua are tree-like, therefore, extending results of Bing and Rogers for homeomorphism and a result of Sturm for the pseudo-circle and pseudo-solenoids. This also provides a partial answer to the question of Lewis whether all continuously homogeneous hereditarily indecomposable continua are homogeneous.
\end{abstract}

\section{Introduction}
A {\it continuum} is a compact and connected metric space. A continuum $X$ is {\it continuously homogeneous} if for any $x,y\in X$ there exists a continuous surjection $f:X\to X$ with $f(x)=y$. If for any such $x$ and $y$ a surjection $f$ as above can be found that is also injective, then $X$ is said to be {\it homogeneous}. Homogeneous planar compacta and continua have recently been classified by Hoehn and Oversteegen \cite{Hoehn}. According to their result, any such continuum is either a circle, pseudo-arc or circle of pseudo-arcs. The result for compacta follows by taking a Cartesian product of one of the above with a finite set or Cantor set. The problem of classifying topologically homogeneous compacta remains a major open problem in topology. Because of the central role that the hereditarily indecomposable pseudo-arc played in the classification of planar homogeneous compacta, it is of high relevance to better understand the interplay between the notions of homogeneity, with its generalizations, and hereditary indecomposability. In this vein, the present paper is motivated by the following open problem (Question 12, \cite{Lewis}).
\begin{question}( \cite{Lewis})
Does there exist a hereditarily indecomposable continuum that is not homogeneous, which is homogeneous with respect to continuous surjections. 
\end{question}
The most natural candidates for such a continuum are continua with dimension greater than one, or those that have a nontrivial shape, since by a result of Rogers none of them is homogeneous \cite{Rogers3}. Here, however, we show that neither of them is continuously homogeneous either.
\begin{theorem}\label{main}
Suppose $X$ is a continuously homogeneous hereditarily indecomposable continuum. Then $X$ is tree-like. 
\end{theorem}
Theorem \ref{main} extends the result of Rogers \cite{Rogers3} who showed that every homogeneous hereditarily indecomposable continuum is tree-like. Earlier, Bing \cite{Bing} showed that no $n$-dimensional ($2\leq n<\infty$) hereditarily indecomposable continuum is homogeneous. Sturm \cite{Sturm} showed that pseudo-circle and pseudo-solenoids are not continuously homogeneous. In \cite{KuperbergGammon} K. Kuperberg and Gammon gave a short proof of the nonhomogeneity of the pseudo-circle. Generalizing their approach to higher dimensions and noninjective maps, and combining it with some ideas of Rogers from \cite{Rogers2}, \cite{Rogers}, in Lemma \ref{cfcn} we give a new criterion for continuous nonhomogeneity. This criterion is very handy for hereditarily indecomposable continua that are not tree-like, but we foresee further applications (with suitable adjustments) in the future to other indecomposable continua. This is summarized in Conjecture \ref{fingen}, pertaining to finitely generated continua. Such a new criterion was needed in the present work, since most proofs of nonhomogeneity of hereditarily indecomposable continua take advantage of Effros Theorem \cite{Effros}, stating that a Polish group $G$ acting transitively on a Polish space $X$ acts on $X$ micro-transitively. However, since to prove results on continuous homogeneity we deal no longer with homeomorphisms but merely continuous surjections, the use of the Effros Theorem is no longer a valid option. The notion of generalized homogeneity of continua, where instead of the action of the group of homeomorphisms on a space, one considers the action of semigroups of certain classes of surjections of the space, was first introduced by J. Charatonik and Bellamy in the late 1970s. Charatonik showed that if $X$ is a circle-like continuum, then $X$ is a solenoid if and only if $X$ is homogeneous with respect to the class of open mappings and each subcontinuum of $X$ is an arc. He asked whether homogeneity with respect to the class of open maps is equivalent to homogeneity for continua \cite{Charatonik1}, \cite{Charatonik2}. This question was independently answered in the negative by Prajs and Seaquist, with their construction of a continuous decomposition of the Sierpi\'nski Carpet into pseudo-arcs \cite{PrajsO1}, \cite{PrajsO2}, \cite{Seaquist1}, \cite{Seaquist2}. Charatonik conjectured that the pseudo-circle might provide another counterexample, the conjecture that was eventually disproved by Sturm, as already mentioned. Earlier, in 1979 Krupski \cite{Krupski1} showed that the class of continuously homogeneous continua contains all Peano continua and arbitrary cones, but the class does not include non-locally connected real or half-ray curves, non-locally connected chainable continua with finitely many arc components, and non-arcwise connected $\lambda$-dendroids with finitely many arc components. 

A connected set $X$ is said to be {\itshape tree-like} if for each $\epsilon>0$ there exists a tree $T$, and a map $f:X\to T$ such that $\operatorname{diam}(f^{-1}(t))<\epsilon$ for each $t\in T$. If $T$ is an arc for each $\epsilon$, then $X$ is said to be {\itshape arc-like}. A continuum $X$ is said to be {\itshape indecomposable} if given subcontinua $A,B\subset X$ such that $A\cup B=X$ we must have $A=X$ or $B=X$. A connected set is said to be {\itshape hereditarily indecomposable} if each subcontinuum is indecomposable. Every continuum is the continuous image of a one-dimensional hereditarily indecomposable continuum \cite{Bellamy}, \cite{Bellamy2}, \cite{HartMillPol}, \cite{MackowiakTymchatyn}. The best known example of a hereditarily indecomposable continuum is a {\itshape pseudo-arc}, first constructed by Knaster in 1922 \cite{Knaster}, which is characterized as the unique hereditarily indecomposable arc-like continuum \cite{Bing59}. Recently, it was also characterized as the only, other than the arc, planar continuum homeomorphic to each of its proper subcontinua \cite{HoehnOversteegen19}, \cite{Moise}. The homogeneity of the pseudo-arc was first proven by Bing \cite{Bing48}. He also showed that most continua are pseudo-arcs, in the sense that for any Euclidean space or Hilbert space $\mathbb{E}$, of dimension at least $2$, the collection of subcontinua of $\mathbb{E}$ which are pseudo-arcs is a dense $G_\delta$ in the hyperspace of subcontinua of $\mathbb{E}$ \cite{BingPacific}. In dynamics, pseudo-arcs appear in the closure of the unstable manifold of dissipative saddle periodic points of arbitrary small $C^r$-perturbations of any $C^r$-diffeomorphism of a 2-manifold that exhibits a homoclinic tangency \cite{Martensen}. Also, there exists a dense $G_\delta$ subset of the set of Lebesgue measure-preserving interval maps $f$ such that the inverse limit of the map is the pseudo-arc \cite{CO}. Besides the pseudo-arc and topology, other hereditarily indecomposable continua also appear in other areas of mathematics. This includes smooth dynamics \cite{Ha}, \cite{KennedyYorke1}, \cite{KennedyYorke2}, \cite{KennedyYorke3}, complex dynamics \cite{Cheritat}, \cite{Herman}, \cite{Rempe-Gillen}, topological dynamics \cite{COJDE},\cite{Minc}, rotation theory \cite{BCJ}, \cite{BO}, isometric theory of Banach spaces \cite{BS}, \cite{Kawamura}, \cite{Rambla}, and Fra\"ise limit theory \cite{Kubis}, \cite{IrwinSolecki}, \cite{Kwiatkowska}.

\section{Hereditarily indecomposable continua}
Given a set $A\subseteq X$ by $\cl_X(A)$ we denote the closure of $A$ in $X$. If $X$ is clear from the context, then we shall just write $\cl(A)$. By $\mathbb{S}^1$ we shall denote the unit circle and $I^\infty$ will stand for the Hilbert cube. We set $\Sigma=\mathbb{R}\times I^\infty$ for the universal covering space of $\mathbb{S}^1\times I^\infty$, with the covering map $\tau$. By $\Sigma_\infty$ we shall denote the one-point compactification of $\Sigma$ by the point $\infty$. For the following fact, see, for example \cite{Hu}, Theorem 16.3.
\newline 
{\bfseries Lemma} [Map Lifting Property]
{\itshape Suppose $X$ is a topological space and $(\tilde{X},\tau)$ is its universal covering space, then given a map $f:X\to X$ there exists a {\itshape lift} $g:\tilde{X}\to \tilde{X}$ such that the following diagram commutes.
\[
\begin{CD}
\tilde{X}@>g>> \tilde{X} \\
@V\tau VV @V\tau VV \\
X @>f>>X
\end{CD}
\]
Additionally, if $f(x)=y$ then $g$ is uniquely determined by the choice of two points $\tilde{x}\in\tau^{-1}(x)$, $\tilde{y}\in\tau^{-1}(y)$ and the setting $g(\tilde{x})=\tilde{y}$\footnote{So the number of such lifts $g$ is related to the number of sheets in the covering space - in fact the number of such lifts is equal to the number of elements in the deck transformation group.}}

Recall that given a continuum $Z$ and a point $z\in Z$, the {\itshape composant} of $z$ in $Z$ is the set of all points $z'\in Z$ for which there exists a proper subcontinuum $Z'\subset Z$ such that $z,z'\in Z'$. \begin{remark}
We shall apply the same definition of a composant to connected sets, even if they are not compact.
\end{remark} A map $f:X\to Y$ is said to be {\it confluent} if for any continuum $K\subset Y$ and any component $C$ of $f^{-1}(K)$ we have $f(C)=K$. Note that any homeomorphism is confluent. Any map onto a hereditarily indecomposable continuum is confluent \cite[Lemma 15]{RogersIll}. 
Even though Bellamy and Lewis showed that the two-point compactification of the universal cover of the pseudo-circle is the pseudo-arc, and thus the cover itself is connected, it is easy to give examples of hereditarily indecomposable continua whose universal covers are not connected. One such an example is a hereditarily indecomposable version of a ray limiting to a circle, in which the universal cover consists of countably many arc-like components. 
\begin{example}\label{ex1}
Let $O$ be the one-point union of $S^1$ and an arc $A=[a,b]$, with midpoint $c$, such that $S^1\cap A=\{a\}$. Let $\varphi:O\to O$ be a piecewise linear map such that 
\begin{itemize}
\item $\varphi(c)=a$, 
\item $\varphi(a)=a$, $\varphi(b)=b$, $\varphi([a,c])=\mathbb{S}^1$, 
\item $\varphi(\mathbb{S}^1)=\mathbb{S}^1$, and 
\item $\varphi|\mathbb{S}^1$ topologically exact.
\end{itemize} 
By Theorem 14 in \cite{KOT}, $\varphi$ can be chosen so that $X=\lim_{\leftarrow}\{O,\varphi\}$ is hereditarily indecomposable. Then the universal cover $\tilde{X}$ of $X$ consists of a countable number of pairwise disjoint hereditarily indecomposable arc-like components, one of which is $\mathcal{C}$, the universal cover of the pseudo-circle $P=\lim_{\leftarrow}\{\mathbb{S}^1,\varphi|\mathbb{S}^1\}$, and all other $\{\mathcal{K}_n:n\in\mathbb{N}\}$ cover the "pseudo-spiral" around $P$. Therefore both the two-point compactification, and one-point compactification of $\tilde{X}$ are decomposable continua. This continuum is not continuously homogeneous, by the result of Sturm stating that the pseudo-circle in not continuously homogeneous \cite{Frank}, and the fact that the pseudo-circle is invariant under any surjective self-map $f$ of $X$ (since for any lift $\tilde{f}$ of $f$ one must have $\tilde{f}(\mathcal{C})\cap \tilde{f}(\mathcal{K}_n)=\emptyset$ for all $n$). 
\end{example}
\begin{example}\label{bco}
In \cite{BCO2} Clark, Oprocha and the first author constructed a family $\Theta$ of one-dimensional hereditarily indecomposable continua, by a so-called  pseudo-suspension method. The spaces resemble mapping tori of Cantor set homeomorphisms, with each arc-component replaced by the universal cover $\mathcal{C}$ of the pseudo-circle, and embed essentially into the solid torus $\mathbb{S}^1\times [0,1]^2$. A universal cover of each such continuum consists of uncountably many copies of $\mathcal{C}$. Consequently, each component of such a universal cover has an indecomposable one-point compactification homeomorphic to a pseudo-arc with two points from two distinct composants identified (see \cite{BL} and \cite[Remark, p.1151]{KuperbergGammon}). Since each subcontinuum of $\mathcal{C}$ is a pseudo-arc, which is acyclic, the covering map is injective on each such subcontinuum. Because every $X\in \Theta$ is hereditarily indecomposable, $X$ is not continuously homeogeneous by Lemma \ref{cfcn}. 
\end{example}
The examples above suggest that if we consider one component of the universal cover at a time, then the one-point compactification is an indecomposable continuum. We confirm this conjecture in the following lemma. 
\begin{lemma}\label{component}
Suppose $X$ is a hereditarily indecomposable continuum that embeds essentially in $\mathbb{S}^1\times I^\infty$. Then the closure of any component of $\tau^{-1}(X)$ in $\Sigma_\infty$ is indecomposable.
\end{lemma}

\begin{proof}
Let $X$ be given as in Lemma \ref{component}, $K$ be such a component and $K_\infty$ be the closure of $K$ in $\Sigma_\infty$. Note that proper subcontinua of $K$ are indecomposable by Theorem 10 in \cite{Rogers3}\footnote{Theorem 10 in \cite{Rogers3} is formulated for the case where $X$ is 1-dimensional and $\tau$ is a universal covering map of a cube with handles; however, the proof easily extends to the general case, which is later observed in \cite{Rogers3} on p.426, and used in \cite{Rogers2}.}. 
\begin{claim}\label{injective}
$\tau|A$ is injective for any continuum $A\subset K$. 
\end{claim} 
\begin{proof}(of Claim \ref{injective})
Indeed, if there were $x,y\in A$ such that $\tau(x)=\tau(y)$ then there would exist a deck transformation\footnote{i.e. $\tau\circ\sigma=\tau$} $\sigma: \mathbb{R}\times I^\infty \to \mathbb{R}\times I^\infty$ such that $\sigma(x)=y$. Note that $\sigma$ is of the form $\sigma(t,x)=(t+j,x)$, for any $(t,x)\in\mathbb{R}\times I^\infty$ and some $j\in\mathbb{Z}$. Consequently, one of the following must hold:
\begin{itemize}
\item[(1)] $\sigma(A)\setminus A\neq \emptyset$ and $A\setminus \sigma(A)\neq \emptyset$, or
\item[(2)] $\sigma (A)\subset A$, or
\item[(3)] $A\subset \sigma(A)$. 
\end{itemize}
In the first case, we get a decomposable subcontinuum $\sigma(A)\cup A$ of $K$, which is in contradiction with the hereditary indecomposability of $K$. In the second case, we would get $\sigma^n(A)\subset A$ for all $n\in\mathbb{N}$, so $A$ is unbounded, leading to a contradiction with compactness of $A$. The same contradiction we get in the third case, by the fact that $\sigma^n(A)\subset A$ for all $-n\in\mathbb{N}$. This shows that $\tau|A$ is injective and completes the proof of Claim \ref{injective}. 
\end{proof}
By the above claim and confluency of $\tau$, it follows that if $p \in K$ and $M$ is a proper subcontinuum of $\tau(K)$ containing $\tau (p)$ then there is a subcontinuum $\hat M \subset K$ containing $p$ so that $\tau (\hat M) = M$.

\begin{claim}\label{connected} Suppose that $A \cup B$ is a decomposable subcontinuum of $K_\infty$. Then $A \cap B$ is connected. In particular, $K_\infty$ is unicoherent.
\end{claim}
\begin{proof}(of Claim \ref{connected}) Suppose $A \cap B \subset K_\infty$ is the union of two disjoint compact sets $E$ and $F$. Without loss of generality, we have $\infty \notin E$. Let $U$ be an open set containing $E$ whose closure does not contain
$\infty$ and does not intersect $F$ . Let $C$ be a component of $A \cap \cl(U)$ intersecting $E$. Then $C$ is a continuum missing points of $B$ that intersects $B$ hence, by hereditary indecomposability, $C \subset B$. But $C$ contains points of $A$ that are not in $A \cap B$ (from the boundary of $U$, by the boundary bumping theorem), which is a contradiction, and the proof of Claim \ref{connected} is complete.
\end{proof}

\begin{claim}\label{infty} Suppose that $A \cup B$ is a decomposable subcontinuum of $K_\infty$. Then $\infty \in A \cap B$.
\end{claim}

\begin{proof}(of Claim \ref{infty}) Suppose $\infty \notin A \cap B$. Without loss of generality, assume $\infty \notin A$. As above, consider an open set $U$ containing $A$ whose closure misses $\infty$ and some point $x$ of $B$. Then let $C$ be the component of $(A \cup B) \cap \cl{U}$ containing $A$. Let $V$ be an open set whose closure misses $\infty$ and $x$; let $D$ be the component of $B \cap \overline{V}$ intersecting $A \cap B$. Then since (by Boundary Bumping Theorem) $D$ intersects $Bd(V)$, $D$ contains a point of $B$ not in $C$. So by hereditary indecomposability, since $D$ intersects $C$ we have $C \subset D$ but $D \subset B$ and $A \subset C$ so $A \subset B$. This contradiction completes the proof of Claim \ref{infty}.
\end{proof}

For the remainder, assume by contradiction that $K_\infty$ is decomposable into continua $A$ and $B$. We know that $\infty \in A \cap B$, and that $A \cap B$ is nondegenerate, since otherwise $K$ would be disconnected. Let $p \in A \cap B$ be such that $p \ne \infty$ and $p$ is a boundary point of $A \setminus A \cap B$. This point exists because otherwise $K$ would not be connected. Let $U$ be an open set containing $p$ such that there exists an open set $V$ containing $\tau (p)$ so that $\tau |_U$ is a homeomorphism onto $V$. Then $U \setminus A\cap B$ is a nonempty open set, by the choice of $p$.

By \cite[p. 112, Theorem 14.5]{Nadler} there exists an order arc $\alpha:[0,2]\to C(A\cap B)$ from $\{p\}$ to $A\cap B$ in the hyperspace of subcontinua of $A\cap B$, such that for each $t \in [0,2]$, $\alpha(t) = A_t$ is a subcontinuum of $A\cap B$, and $A_t \subset A_s$ for $t < s$, with $A_0 = \{p\}$ and $A_2=A\cap B$. Note that $\infty\in A_2$ and the set $\{t\in[0,2]:\infty\in A_t\}$ is a closed subinterval of $[0,2]$ (otherwise there would exist a $t'$ such that $A_{t'}=\cap_{t>t'}A_t$ with $\infty\notin A_{t'}$, but $\infty\in A_t$ for $t>t'$). Therefore, there exists a $t_0\in[0,2]$ such that $t_0=\min\{t\in[0,2]:\infty\in A_t\}$, and without loss of generality we can assume that $t_0=1$. 

Consider $M$, the closure of $\cup_{t < 1} \tau (A_t)$; if this is a proper subcontinuum $M$ of $X$ then there is a subcontinuum $\hat M$ of $K$ containing $p$ and mapping onto $M$; but then $\hat M = A_1$ (since by hereditary indecomposability of $X$ each $\tau (A_t) \subset M$) and $\infty \notin \hat M$ which is a contradiction. So $\cup_{t < 1} \tau ( A_t)$ must be dense in $X$ (for otherwise its closure would be a proper subcontinuum of $\tau(K)$). Consequently, there is an $r<1$ so that $\tau (A_r)$ intersects the open set $\tau (U - A\cap B)\subset V$. But this is a contradiction, since $A_r\subset A\cap B$ for all $r$ and $\tau|U$ is a homeomorphism. The proof of Lemma \ref{component} is complete.
\end{proof}
\section{Continuous Homogeneity Criterion}
Before we proceed to the proof of Theorem \ref{main} we make the following observations.
\begin{proposition}\label{composants}
Suppose $X$ and $Y$ are connected sets. Let $f:X\to Y$ be a confluent surjection. If $C$ is a composant of $X$ then either $f(C)=Y$ or $f(C)$ is a composant of $Y$. 
\end{proposition}
\begin{proof}
Suppose $f(C)\neq Y$ and $p\in C$. Let $D$ be the composant of $Y$ that contains $f(p)$. Consider any other point $y$ of $D$. There is a proper subcontinuum $I$ of $Y$ containing $f(p)$ and $y$. By the confluence of $f$ the component $J$ of the preimage of $I$ containing $p$ is mapped onto $I$. Since $C$ does not map onto all of $Y$, the continuum $J$ is a proper subcontinuum of $C$, so $y$ is in the image of $C$. Consequently $C$ maps onto the composant $D$.
\end{proof}
\begin{remark}\label{CY}
Note that in the setting of Proposition \ref{composants}, if $f(C)=Y$ and $Y$ has at least 2 pairwise disjoint composants, then there exists a connected proper subset of $W\subset C$ such that $f(W)=Y$. Moreover, if $Y$ is compact then $W$ is a continuum. Indeed, since $f(C) = Y$ there exists a point $q\in C$ such that $f(q)\notin D$.  Since $C$ is a composant there is a proper subcontinuum $W$ of $C$ containing $p$ and $q$.  So $W$ has to map onto $Y$.
\end{remark}
\begin{lemma}\label{tnont}
No hereditarily indecomposable continuum with a trivial 1-st \v Cech cohomology group can be mapped onto a hereditarily indecomposable continuum with a nontrivial 1-st \v Cech cohomology group. 
\end{lemma}
\begin{proof}
The proof is identical to that of \cite[Theorem 16]{RogersIll}. Namely, let $X$ and $Y$ be hereditarily indecomposable, with $\check H^{1}(X;\mathbb{Z})=0$ and ${\check{H}}^{1}(Y;\mathbb{Z})\neq 0$. By contradiction, suppose there exists a surjective map $f:X\to Y$. By the result of Rogers \cite[Lemma 15]{RogersIll}, $f$ is confluent. By a result of Lelek \cite[Corollary 2]{Lelek}, $f$ induces a monomorphism $\check f:{\check{H}}^{1}(Y;\mathbb{Z})\to{\check{H}}^{1}(X;\mathbb{Z})$. Since ${\check{H}}^{1}(X;\mathbb{Z})$ is trivial, and ${\check{H}}^{1}(Y;\mathbb{Z})$ is not, we obtain a contradiction. 
\end{proof}
\begin{lemma}(Continuous Homogeneity Criterion)\label{cfcn}
Let $X$ be a hereditarily indecomposable continuum. Suppose that $X$ embeds essentially into $\mathbb{S}^1\times I^\infty$. Then $X$ is not continuously homogeneous. 
\end{lemma}
\begin{proof}(of Lemma \ref{cfcn})
Suppose $f: X\to X$ is a surjective map. Since $X$ is a closed subset of the ANR $\mathbb{S}^1\times I^\infty$, $f$ extends to a map $f: \mathbb{S}^1\times I^\infty\to \mathbb{S}^1\times I^\infty$, by first extending to the closure of an open neighborhood $U$, and then composing with a retraction from $\mathbb{S}^1\times I^\infty$ onto $\cl(U)$. Consequently $f$ has a lift $\tilde{f}:\Sigma\to \Sigma$; i.e. $\tau\circ \tilde{f}=f\circ \tau $. Since $f|X$ is a surjection, by the result of Lelek, it induces an automorphism of the 1-st \v Cech cohomology group of $X$, and consequently $\tau^{-1}(f(X))$ is unbounded in $\Sigma$. But $\tau^{-1}(f(X))=\tilde{f}(\tau^{-1}(X))$, hence $\tilde{f}$ extends to $\Sigma_\infty$ by setting $\tilde f(\infty)=\infty$.

Let $Z\subset X$ be a minimal (with respect to the inclusion) continuum with nontrivial 1-st \v Cech cohomology; i.e. any subcontinuum $T\subset Z$ has trivial 1-st \v Cech cohomology. Let $K$ be a component of $\tau^{-1}(Z)$ and $K_\infty$ be the closure of $K$ in $\Sigma_\infty$. Since $K_\infty$ is indecomposable by Lemma \ref{component}, it has uncountably many pairwise disjoint composants. Suppose that $\tilde{x},\tilde{y}\in K$, and that for $x=\tau(\tilde{x})$ and $y=\tau(\tilde{y})$ we have $f(x)=y$. 
\begin{claim}\label{sub}
$f(Z)\subset Z$
\end{claim}
 \begin{proof}(of Claim \ref{sub})
 By the fact that $f(Z)\cap Z\neq\emptyset$ and hereditary indecomposability of $X$ we must have that $f(Z)\subset Z$ or $Z\subset f(Z)$. If it were true that $Z\subsetneq f(Z)$, then there would have been a subcontinuum $N\subsetneq Z$ with $f(N)=Z$, contradicting Lemma \ref{tnont}.      
 \end{proof}
 By Map Lifting Property, we can choose $\tilde{f}$ to be the unique lift satisfying $\tilde{f}(\tilde{x
})=\tilde{y}$. Then $\tilde f(K)=K$ since $\tilde f(K)\cap K\neq\emptyset$ and $K$ is a component of $\tau^{-1}(Z)$. It follows that $\tilde f(K_\infty)=K_\infty$.
\begin{claim}\label{Z}
 $f(Z)=Z$   
\end{claim}
\begin{proof}(of Claim \ref{Z})
 We have already shown that $f(Z)\subset Z$. If $f(Z)\subsetneq Z$ then $f(Z)$ has trivial 1-st \v Cech cohomology and consequently $\tau^{-1}(f(Z))$ is a compact subset of $\Sigma$. However, $\tilde f(K_\infty)$ must be a subcontinuum of $\tau^{-1}(f(Z))$, contradicting $\infty \in \tilde f(K_\infty)$.    
\end{proof}
\begin{claim}\label{notimage}
If $C$ is a composant of $Z$ then $f(C)$ is also a composant of $Z$.  
\end{claim}
\begin{proof}(of Claim \ref{notimage})
By Proposition \ref{composants} either $f(C)$ is a composant or $f(C)=Z$. But by Remark \ref{CY}, the choice of $Z$ and Lemma \ref{tnont}, we get $f(C)\neq Z$, so $f(C)$ is a composant. 
\end{proof}

Now we shall make a more specific choice of $x,y\in Z$. Let $C_\infty$ be the composant of $\infty$ in $K_\infty$ and $\tilde{C}$ be another composant in $K_\infty$. Pick $\tilde{x}\in C_\infty$ and $\tilde{y}\in \tilde{C}$ and let $x=\tau(\tilde{x})$ and $y=\tau(\tilde{y})$.

\begin{claim}\label{empty}
$\tau(C_\infty\setminus\{\infty\})\cap\tau(\tilde C)=\emptyset$.
\end{claim}
\begin{proof}(of Claim \ref{empty})
As noted in \cite[proof of Theorem 2]{KuperbergGammon}, if for some $v$, the composant $C_\infty$ intersects the fiber $\tau^{-1}(v)$, then $\tau^{-1}(v)\subset C_\infty$, since each deck transformation extends to $\Sigma_\infty$ by letting $\infty$ to be a fixed point. If there exists a $z\in \tau(C_\infty\setminus\{\infty\})\cap\tau(\tilde C)$ then $\tau^{-1}(z)\subset C_\infty$ and $\tau^{-1}(z)\cap \tilde C\neq\emptyset$. But $C_\infty\cap\tilde C=\emptyset$, which results in a contradiction.
\end{proof}
\begin{claim}\label{lie}
$\tau(C_\infty\setminus\{\infty\})$ and $\tau(\tilde C)$ lie in different composants of $Z$.
\end{claim}
\begin{proof}(of Claim \ref{lie})
By contradiction, suppose that there exists a proper subcontinuum $P\subset Z$ that intersects both $\tau(C_\infty\setminus\{\infty\})$ and $\tau(\tilde C)$. Since $P$ must have a trivial 1-st \v Cech cohomology group, and $\tau^{-1}(P)\cap C_\infty\neq\emptyset$, by the argument given for fibers of points in the proof of Claim \ref{empty}, we must have $\tau^{-1}(P)\subset C_\infty$. This is a contradiction with the fact that $\tau^{-1}(P)$ must also intersect $\tilde C$.
\end{proof}
To complete the proof of Lemma \ref{cfcn}, let us assume that $f(x)=y$ and let $\tilde{f}$ be the unique lift with $\tilde{f}(\tilde{x})=\tilde{y}$. We let $\tilde{W}\subset C_\infty$ to be a proper subcontinuum that joins $\infty$ and $\tilde{x}$. Then $\tilde{f}(\tilde{W})$ is a subcontinuum of $K_\infty$ joining $\infty$ and $\tilde{y}$. Since $\tilde{y}\notin C_\infty$ it follows that $\tilde{f}(\tilde{W})=K_\infty$. Consequently $\tilde{f}(C_\infty\setminus\{\infty\})=K$. Therefore, if $C_x$ and $C_y$ are composants in $Z$ of $x$ and $y$ respectively, then $f(C_x)$ intersects both $C_x$ and $C_y$. But this is a contradiction with Claim \ref{notimage}.
This contradiction completes the proof of Lemma \ref{cfcn}.
\end{proof}
Now we are in the position to prove that continuously homogeneous hereditarily indecomposable continua are tree-like. 
\begin{proof}(of Theorem \ref{main})
Let $X$ be a continuously homogeneous hereditarily indecomposable continuum and by contradiction suppose $X$ is not tree-like. Then there exists an essential map $\varphi:X\to\mathbb{S}^1$. Indeed, if $X$ is 1-dimensional then for the first \v Cech cohomology group with integral coefficients we have $H^1(X)\neq 0$, and this is equivalent to the existence of the essential map $\varphi$. If $X$ is of dimension at least $2$ then such a map exists by a result of Krasinkiewicz; see Corollary 4, p. 510 in \cite{Krasinkiewicz}. This map gives rise to an essential embedding of $X$ in $\mathbb{S}^1\times I^\infty$. Indeed, it is enough to consider the embedding $e:X\to \mathbb{S}^1\times I^\infty$ given by $e(t)=(\varphi(t),\phi(t))$, where $\phi:X\to I^\infty$ is an embedding of $X$ into $I^\infty$. Now we can apply Lemma \ref{component} and Lemma \ref{cfcn}. The proof is complete. 
\end{proof}
\begin{remark}
Note that if $X$ is tree-like, then there does not exist an essential map $\varphi:X\to\mathbb{S}^1$. In particular, if $X$ is embedded in $\mathbb{S}^1\times I^\infty$ then $\tau^{-1}(X)$ consists of countably many homeomorphic copies of $X$ in $\Sigma$, and the above proof is not applicable. 
\end{remark}
We conjecture that there is a large and quite natural class of continua that lack continuous homogeneity. Suppose $X$ is a continuum. $X$ is said to be \textit{finitely generated} (cf. \cite{JP}) if either, 
\begin{enumerate}
\item $X$ does not embed essentially into $\mathbb{S}^1\times I^\infty$, or
\item $X$ embeds essentially into $\mathbb{S}^1\times I^\infty$ and there exists a continuum $X^*$ in $\Sigma$ such that $\tau(X^*)=X$. 
\end{enumerate}
Note that if $f:\mathbb{S}^1\to\mathbb{S}^1$ has the property that for every open set $U$ there is an $m$ such that $f^m(U)=\mathbb{S}^1$ (i.e. $f$ is topologically exact) then $\lim_{\leftarrow}\{\mathbb{S}^1,f\}$ is not finitely generated. This follows from the proof of Theorem 4.3 in \cite{BO2}. Similarly, if a circle-like continuum is self-entwined then it is not finitely generated \cite{Rogers_self}.
\begin{conjecture}\label{fingen}
Any continuously homogeneous indecomposable planar continuum is finitely generated.
\end{conjecture}
Returning to hereditarily indecomposable continua, recall that Kennedy and Rogers \cite{KennedyRogers} showed that the action of the homeomorphism group of the pseudo-circle has uncountably many orbits, each of which is a union of uncountably many composants. Their result strenghtened the original proofs of nonhomogeneity of the pseudo-circle by Fearnley \cite{Fearnley} and Rogers \cite{Rogers}. Cook showed the existence of a hereditarily indecomposable continuum that admits the identity as the only homeomorphism \cite{Cook}; see also \cite{Pol}. Bing showed \cite{Bing} that in every hereditarily indecomposable $n$-dimensional continuum ($n>1$) there is a composant that contains only $n$-dimensional continua (see also \cite{PolRenska}). Apart from these results, the pseudo-arc and pseudo-circle, it seems that little is known about the degree of homogeneity of hereditarily indecomposable continua in general. In particular, the following question from \cite{Boronski} remains open. 
\begin{question}
For a hereditarily indecomposable continuum $X$ must the homeomorphism group $\operatorname{Homeo}(X)$ have either one or infinitely many orbits?
\end{question}
If there are infinitely many orbits, it is also unknown how many of them there are, even for the pseudo-circle; see \cite{KennedyRogers}.
\begin{question}
For a hereditarily indecomposable continuum $X$, if the homeomorphism group $\operatorname{Homeo}(X)$ has infinitely many orbits, must it have $2^{\aleph_0}$ many orbits? What if $X$ is the pseudo-circle?
\end{question}
\section{Acknowledgements}
This work was supported in part by National Science Centre, Poland (NCN), grant "Homogeneity and Minimality in Compact Spaces" no. 2015/19/D/ST1/01184. Part of the work was carried out during first author's visit in Erie, PA, in July of 2018, thanks to the subsidy for institutional development IRP201824 ``Complex topological structures'' from University of Ostrava. 
\bibliographystyle{plain}

\end{document}